\documentclass[12pt,reqno]{amsart}
\usepackage{amsmath,amsthm,amssymb,amsfonts,amscd}
\usepackage{mathrsfs}
\usepackage{bbding}
\usepackage{graphicx}
\usepackage{hyperref}
\usepackage[top=2.5cm, bottom=2.5cm, left=2.5cm, right=2.5cm]{geometry}
\usepackage{color}
\usepackage{xcolor}
\usepackage[english]{babel}
\usepackage{mathtools}
\hypersetup{
    colorlinks,
    linkcolor={red!50!black},
    citecolor={blue!50!black},
    urlcolor={blue!80!black}
}

\numberwithin{equation}{section}

\theoremstyle{plain}
\newtheorem{theorem}[equation]{Theorem}
\newtheorem*{theorem*}{Theorem}
\newtheorem{lemma}[equation]{Lemma}
\newtheorem{corollary}[equation]{Corollary}

\theoremstyle{definition}

\theoremstyle{remark}
\newtheorem{remark}[equation]{Remark}

\renewcommand{\Re}{\operatorname{Re}}

\newcommand{\sym}{\operatorname{sym}}

\newcommand{\bs}{\backslash}

\renewcommand{\mod}{\operatorname{mod}\,}

\newcommand{\dd}{\hspace*{0.1em}\mathrm{d}}% rectified d in integrals

\begin{document}

\title{Mixed moment of $GL(2)$ and $GL(3)$ $L$-functions in the level
aspect}
\author{Zhiwei Liu}
\address{Data Science Institute and School of Mathematics \\ Shandong University \\ Jinan \\Shandong 250100 \\China}
\email{zwliu@mail.sdu.edu.cn}
\date{}

\begin{abstract}
In this paper, we establish an asymptotic formula of mixed moments of $L(1/2,f)$ and $L(1/2,\sym^2f)$ for cusp forms $f$ of weight $k$, prime level $q$ and real primitive nebentypus $\chi$ modulo $q$ in the level aspect with a power-saving error term. As an application, we deduce the simultaneous non-vanishing of $L(1/2, f)$ and $L(1/2, \sym^2 f)$ for sufficiently large prime levels.  
\end{abstract}

\thanks{This work was partially supported by the National Key R\&D Program of China (No. 2021YFA1000700) and 
the Scientific Research Innovation Capability Support Project for Young Faculty (No. SRICSPYF-ZY2025158).}
\maketitle
\section{Introduction} \label{sec:intr}
The study of moments of $L$-functions is one of the central topics in modern analytic number theory. The asymptotic behavior of these moments, particularly the leading order main terms, is predicted by the random matrix theory \cite{KS2000}, the theory of multiple Dirichlet series \cite{DGH2003} or the recipe in \cite{KeatingSnaith2005}. Estimating the moments of families of $L$-functions provides deep insight into the distribution of their values and the underlying symmetry types, and there are many applications, such as the non-vanishing problem and the subconvexity problems, see e.g. \cite{duke1995,DFI2002,IS2000}.

In this paper, we are interested in the mixed moment of $L(1/2,f)$ and $L(1/2,\sym^2 f)$ for cusp forms with real primitive nebentypus modulo a prime $q$, and we deduce a simultaneous non-vanishing result.

Let $k\ge 2$ be an integer, $q\geq 1$ be a large prime and $\chi$ be a real primitive character to modulus $q$ satisfying $\chi(-1)=(-1)^{k}$. Let $\mathcal{S}_{k}(q,\chi)$ be the Hilbert space of holomorphic cusp forms of weight $k$ with respect to the congruence subgroup $\Gamma_0(q)$ and nebentypus $\chi$. Since $\chi$ is primitive, we can choose a normalized orthogonal basis $\mathcal{B}_k(q,\chi)$ consisting of newforms. For $f\in \mathcal{B}_k(q,\chi)$, let $ \lambda_f(n)$ be the $n$-th Hecke eigenvalue of $f$, the Hecke $L$-function is defined by
\begin{align*}
    L(s,f)=\sum_{n=1}^{\infty} \lambda_f(n)n^{-s},\ \Re (s)>1.
\end{align*}
Following \cite{blomer2008}, for $\Re (s)>1$ the symmetric square $L$-function is given by
\begin{equation*}
    L(s,\sym^2f)=\frac{L(s,f\otimes f)}{L(s,\chi)}=\zeta^{(q)}(2s)\sum_{m\mid q^{\infty}}\tau(m)\sum_{(n,q)=1}\frac{\lambda_f((mn)^2)}{(mn)^s},
\end{equation*}
where $\zeta^{(q)}(2s):=\prod_{p\nmid q}(1-p^{-2s})^{-1}$. Define
\begin{equation*}
    \omega_f^{-1}=\frac{\Gamma(k-1)}{(4\pi)^{k-1}\left\| f  \right\|^2 }.
\end{equation*}
to be the spectral weight. By a standard argument (see e.g. \cite{topics,DFI2002,HL1994}), we have
\begin{align*}
    \frac{1}{q(\log q)^3}\ll \omega_f^{-1}\asymp \frac{1}{qL(1,\sym^2f)}\ll \frac{1}{q^{1-\varepsilon}}.
\end{align*}

Now we are ready to state our main result.
\begin{theorem}\label{main thm}
    Let $q$ be a prime and let $k\geq 4$ be an integer. Let $\nu=0$ if $k$ is even and $\nu=1$ if $k$ is odd. Then
    \begin{equation}
        \sum_{f\in \mathcal{B}_k(q,\chi)}\omega_f^{-1}L(1/2,f)L(1/2,\sym^2f)=\frac{\zeta(3/2)}{2}\log q +C_k+O_{k,\varepsilon}(q^{-\frac{1}{4}+\delta_k+\varepsilon}),
    \end{equation}
where $\delta_k=\frac{1}{8(k-3/2)}$ and $C_k=\zeta(3/2)\left( 2\gamma-\log 2-\frac{3}{2}\log \pi+\frac{{\zeta}'(3/2)}{\zeta(3/2)}+\psi\left(k-\frac{1}{2}\right)+\frac{1}{2}\psi\left(\frac{3}{4}-\frac{\nu}{2}\right)\right)$ is a constant and the implied constant in the error term depends only on $k$ and $\varepsilon$.

Here $\gamma$ denotes the Euler constant and $\psi(s)=\Gamma'(s)/\Gamma(s)$ denotes the digamma function.
\end{theorem}
\begin{remark}
    Since $\delta_k$ is decreasing in $k$, the error term is uniformly $O(q^{-\frac{1}{5}+\varepsilon})$ for all $k\geq 4$.
\end{remark}

A direct application is the following simultaneous non-vanishing result.
\begin{corollary}
    Given an integer $k\geq 4$, there exists a computable constant $q_k$ such that for any prime $q>q_k$, there exists a cusp form $f\in \mathcal{B}_k(q,\chi)$ such that
    \begin{equation}
        L(1/2,f)L(1/2,\sym^2f)\ne 0.
    \end{equation}
\end{corollary}

 Moments of $GL(2)$ $L$-functions in the level aspect have been studied extensively. For example, Duke, Friedlander and Iwaniec \cite{DFIIII} studied the amplified fourth moments and got a subconvexity result. Kowalski, Michel and Vanderkam \cite{KMV2000} studied the mollified fourth moments and also obtained the subconvexity result. For moments of symmetric square $L$-functions in the level aspect, Iwaniec and Michel \cite{iwaniecmichel2001} proved the upper bound for the second moment with forms having trivial nebentypus. Blomer \cite{blomer2008} gave asymptotic formulas for the twisted first moment and twisted second moment with forms having real primitive nebentypus. For mixed moments of $GL(2)$ and $GL(3)$ $L$-functions in the level aspect, Munshi and Sengupta \cite{Munshi2018} proved an asymptotic formula for the first moment with forms having trivial nebentypus. In the weight aspect, there is a series of work by  Balkanova, Bhowmik, Frolenkov, and Raulf, see \cite{BBFR2019,BBFR2020}. Huang and Li \cite{huang2024mixedmomentsrmgl2} also proved an asymptotic formula for the mixed moments in the spectral aspect.

The question of simultaneous non-vanishing has also received a lot of attention. For example, Ramakrishnan and Rogawski \cite{RR2005} proved a simultaneous non-vanishing result for $GL(2)\times GL(1)$ and $GL(2)$ $L$-functions. More recent work includes \cite{Lsc2015,Munshi2018,MG2023,MV2002,Kummari2024}.

\subsection{Sketch of the proof}
The proof of Theorem \ref{main thm} proceeds as follows. Applying the approximate functional equation for $L(1/2,f)$ and $L(1/2,\sym^2f)$ and the Petersson trace formula, the mixed moment splits as
\begin{align*}
\sum_{f}\omega_f^{-1}L(1/2,f)L(1/2,\sym^2f)=\mathcal{D}_1+\mathcal{D}_2+2\pi i^{-k}(\mathcal{O}_1+\mathcal{O}_2),
\end{align*}
where $\mathcal{D}_1$,$\mathcal{D}_2$ are diagonal terms and $\mathcal{O}_1$,$\mathcal{O}_2$ are off-diagonal terms.

For $\mathcal{D}_1$, we shift contours carefully in the Mellin inversion formulas for the weight functions. 
A crucial technical point is that we choose the weight function $G_1(u)$ in approximate functional equation of $L(1/2,f)$ to have double zero at $u=-1/4$ to avoid the appearance of the secondary main term of size $O(q^{-1/8+\varepsilon})$. By a standard contour shift argument, we pick up the residue at $u=v=0$ to obtain the main term $\frac{\zeta(3/2)}{2}\log q +C_k$. The other diagonal term $\mathcal{D}_2$ is negligible.

The off-diagonal term $\mathcal{O}_1$ is handled directly by Weil's bound and the estimate $J_{k-1}(x)\ll x^{k-1}$. The main difficulty lies in $\mathcal{O}_2$. We need to bound the following sum 
\begin{align*}
    (NM)^{-1/2}\sum_{m\asymp M}\sum_{n\asymp N}\sum_{c\asymp C}\frac{S_{\chi}(mq,n^2;cq)}{cq}J_{k-1}\left(\frac{4\pi \sqrt{m}n}{c\sqrt{q}}\right),
\end{align*}
where $M,N,C$ are dyadic parameters and $M,N\ll q^{1/2+\varepsilon}$. Then we apply the Poisson summation formula in the $m$-sum and get
\begin{align*}
N^{-1/2}M^{1/2}\sum_{n\asymp N}\sum_{c\asymp C}\frac{1}{c^2q}&\sum_{m\in \mathbb{Z}}\sum_{a\ (\mod c)}S_{\chi}(aq,n^2;cq)e\left ( \frac{am}{c} \right ) \\
    &\times\int_{\mathbb{R}} \mathcal{V}(x)J_{k-1}\left(\frac{4\pi \sqrt{M}n\sqrt{x}}{c\sqrt{q}}\right)e\left ( -\frac{mMx}{c} \right ) dx.
\end{align*}

We introduce a small parameter $\delta>0$ to separate the $c$-sum into two ranges. For large $c$, i.e. $C\gg q^{\delta-1/2}\sqrt{M}N$, the Bessel function is small and trivial estimation together with Weil's bound gives an error of $O(q^{-1/8-\delta(k-5/2)+\varepsilon})$. For small $c$, i.e. $C\ll q^{\delta-1/2}\sqrt{M}N$, we evaluate the character sum explicitly and obtain square root cancellation. The analytic part is then treated by the stationary phase method. We further distinguish two subcases depending on whether $\frac{\sqrt{M}N}{c\sqrt{q}}\gg q^{\varepsilon}$ or not. In the former range the stationary phase analysis yields $O(q^{-1/4+\varepsilon})$, while in the latter it yields $O(q^{-1/4+\delta+\varepsilon})$; the dominant contribution is the latter.

Equating the errors from the large and small $c$ regimes by choosing $\delta=\delta_k:=\frac{1}{8(k-3/2)}$ produces the final bound $O(q^{-1/4+\delta_k+\varepsilon})$, which implies a uniform error term $O(q^{-\frac{1}{5}+\varepsilon})$ for all $k\geq 4$. The simultaneous non-vanishing result then follows directly.

\subsection{Structure of the paper}
The paper is organized as follows. In \S \ref{Preliminaries}, we review necessary preliminaries on automorphic forms and $L$-functions.
 In \S \ref{Applying the Petersson trace formula}, we begin our proof with applying the approximate functional equations and the Petersson trace formula to decompose the mixed moment into diagonal and off-diagonal terms. The diagonal term is evaluated in \S \ref{The diagonal term} by contour integration, yielding the main term. The off-diagonal contributions are bounded in \S \ref{The off-diagonal term} using Weil's bound, Poisson summation, and stationary phase analysis. The proof is completed in \S \ref{Proof of main thm} by optimizing the free parameter to obtain the power-saving error term.

\section{Preliminaries}\label{Preliminaries}
In this section, we collect the necessary definitions, notations, and auxiliary results that will be used throughout the paper.

\subsection{Cusp forms with nontrivial nebentypus}
 Let $e(x)=\exp(2\pi ix)$. For $f\in \mathcal{B}_k(q,\chi)$, we have the following Fourier expansion:
\begin{align}
    f(z)=\sum_{n\geq 1}\lambda_{f}(n)n^{\frac{k-1}{2}}e(nz),
\end{align}
where $\lambda_{f}(n)$ are Hecke eigenvalues satisfying Hecke relations:
\begin{equation}\label{heckerelation}
    \lambda_{f}(m)\lambda_{f}(n)= \sum_{ d\mid (m,n)}\chi(d)\lambda_{f}\left(\frac{mn}{d^2}\right)
\end{equation}
for all $m,n \geq 1$ and
\begin{equation}\label{dualoflambda}
    \overline{\lambda_{f}(n)}=\overline{\chi(n)}\lambda_{f}(n)
\end{equation}
for $(n,q)=1$.

For $f\in \mathcal{B}_k(q,\chi)$, the Hecke $L$-function $L(s,f)$ is entire. Let
\begin{align*}
    L_{\infty}(s,f)=(2\pi)^{-s}\Gamma\left(s+\frac{k-1}{2}\right).
\end{align*}
Let $\bar{f}$ denote the dual form with Fourier coefficients $\lambda_{\bar{f}}(n) = \overline{\lambda_f(n)}$. The completed $L$-function 
\begin{align*}
\Lambda(s,f)=q^{-\frac{s}{2}}L_{\infty}(s,f)L(s,f)
\end{align*}
satisfies the functional equation
\begin{equation}
    \Lambda(s,f)=\varepsilon_f\Lambda(1-s,\bar{f}),
\end{equation}
where $\varepsilon_f=i^{k}\overline{\lambda}_f(q)\overline{\tau}_{\chi}q^{-\frac{1}{2}}$ is the root number, and $\tau_{\chi}=\sum_{x\ (\mod q)}\chi(x)e\left(\frac{x}{q}\right)$ is the Gauss sum.

Our analysis is based on the Petersson trace formula.
\begin{theorem}\cite[Theorem 14.5]{ikanalytic}
    For $m,n\geq 1$, we have
    \begin{equation}
        \frac{\Gamma(k-1)}{(4\pi)^{k-1}}\sum_{f\in \mathcal{B}_k(q,\chi)}\frac{\overline{\lambda_{f}(m)}\lambda_{f}(n)}{\left \| f  \right \|^2 }=\delta_{m,n}+2\pi i^{-k}\sum_{q\mid c} \frac{1}{c}S_{\chi}(m,n;c)J_{k-1}\left(\frac{4\pi \sqrt{mn}}{c}\right),
    \end{equation}
    where 
    \begin{align*}
        S_{\chi}(m,n;c)=\sum_{\substack{d\ (\mod c)\\ (d,c)=1}}\chi(d)e\left(\frac{md+n\bar{d}}{c}\right)
    \end{align*}
     is the twisted Kloosterman sum for $q\mid c$, and 
\begin{align*}
    \left \| f  \right \|^2=\iint_{\Gamma_0(q)\bs\mathbb{H}}\left | f(z) \right | ^2y^k\frac{dxdy}{y^2}
\end{align*}
is the square of the norm of $f$.  
\end{theorem}

\subsection{Symmetric square $L$-functions}
The symmetric square $L$-function $L(s,\sym^2f)$ is entire in the complex plane except a possible simple pole at $s=1$. Let $L_{\infty}(s,\sym^2f)=2^{-s}\pi^{-\frac{3}{2}s}\Gamma(s+k-1)\Gamma\left(\frac{s+1-\nu}{2}\right)$, where
\begin{align*}
    \nu=\begin{cases}
      0,    &k\ \ \text{even}\\
      1,&k\ \ \text{odd}
    \end{cases}
\end{align*}
The completed $L$-function 
\begin{equation}
    \Lambda(s,\sym^2f)=q^{\frac{s}{2}}L_{\infty}(s,\sym^2f)L(s,\sym^2f)
\end{equation}
satisfies the functional equation \cite{blomer2008}
\begin{equation}
    \Lambda(s,\sym^2f)=\overline{ \Lambda(1-\bar{s},\sym^2f)}.
\end{equation}

\subsection{Approximate functional equations}
To study the central values $L(1/2,f)$ and $L(1/2,\sym^2f)$, we employ the approximate functional equations, which express these values as essentially finite sums with rapidly decaying weight functions.
By standard contour shift one can prove the following approximate functional equations. 
\begin{lemma}
    We have
\begin{equation}
    L(1/2,f)=\sum_{m=1}^{\infty}\frac{\lambda_f(m)}{m^{1/2}}V\left ( \frac{m}{\sqrt{q}}  \right ) +\varepsilon_f\sum_{m=1}^{\infty}\frac{\overline{\lambda_f(m)}}{m^{1/2}}V\left ( \frac{m}{\sqrt{q}}  \right ),
\end{equation}
where 
\begin{equation}    
V(y)=\frac{1}{2\pi i}\int_{(1)}\frac{L_{\infty}(1/2+u,f)}{L_{\infty}(1/2,f)}\frac{G_1(u)}{u}y^{-u},
\end{equation}
and $G_{1}(u)$ is a holomorphic even function in $\left |\Re u  \right | \leq 1$ satisfying: $G_1(u)$ has double zero at $u=-1/4$, $G_1(0)=1$ and $G_1(u)\ll (1+u)^{-k}e^{\frac{\pi}{2}\left |  u\right | }$, e.g. $G_1(u)=(1-16u^2)^2$.
Moreover, for any $a\geq 0$ we have
\begin{equation*}
    V(y)=1+O_{k}(y^{\frac{k-1}{2}}),
\end{equation*}
and
\begin{equation*}
    V^{(a)}(y)\ll_{a,A,k} y^{-a}(1+y)^{-A}.
\end{equation*}
\end{lemma}
\begin{proof}
    See \cite[Theorem 5.3]{ikanalytic} and \cite[Lemma 3.1]{DFIIII}.
\end{proof}
\begin{remark}
    We need $G_1(u)$ to have double zero at $u=-1/4$ to cancel the double pole of $\zeta^{(q)}(3/2+2u)H(u)$ so that we could avoid the occurrence of the secondary main term of size $O(q^{-1/8+\varepsilon})$. 
\end{remark}

\begin{lemma}\label{afe of sym}
    We have
    \begin{equation}
        L(1/2,\sym^2f)=2\sum_{m\mid q^{\infty}}\tau(m)\sum_{(n,q)=1}\frac{\lambda_f((mn)^2)}{\sqrt{mn}}W\left(\frac{mn}{\sqrt{q}}\right),
    \end{equation}
    where
    \begin{equation*}
        W(y)=\frac{1}{2\pi i}\int_{(1)}\frac{L_{\infty}(1/2+u,\sym^2f)}{L_{\infty}(1/2,\sym^2f)}\zeta^{(q)}(1+2u)\frac{G_2(u)}{u}y^{-u}du,
    \end{equation*}
    where $G_2(u)=1-4u^2$.
    Moreover, for any $B,\varepsilon> 0$ we have 
    \begin{equation}\label{ub of W}
        W(y)\ll_{B,\varepsilon, k} \left ( 1+y^{-1} \right ) \left (  1+y\right ) ^{-B}.
    \end{equation}
\end{lemma}
\begin{proof}
    See \cite[Theorem 5.3]{ikanalytic} and \cite[Ch.2]{blomer2008}.
\end{proof}

Notice that if $m>1$, then $m\mid q^{\infty}$ yields $m\geq q$. Thus for $n\geq 1$, we have $\frac{mn}{\sqrt{q}}\geq \sqrt{q}$. Choosing $B$ above sufficiently large in \eqref{ub of W}, the contribution from $m>1$ is negligible:
\begin{equation}
L(1/2,\sym^2f)=2\sum_{n=1}^{\infty}\frac{\lambda_f(n^2)}{\sqrt{n}}W\left(\frac{n}{\sqrt{q}}\right)+O(q^{-1000}).
\end{equation}

\subsection{Bessel functions}
To handle the oscillatory integrals arising in \S \ref{The off-diagonal term}, we use the following decomposition of $J_{\kappa}(x)$, whose asymptotic behavior is well understood, see \cite[p.206]{watson1995} and \cite[Ch.4]{iwaniecmichel2001}. 
\begin{lemma}\label{besselfunction}
We have
    \begin{equation}
        J_{\kappa}(x)=e^{ix}W(x)+e^{-ix}\overline{W}(x)
    \end{equation}
    where 
    \begin{align*}
        W(x)=\frac{e^{i((\pi \kappa/2)-\pi/4)}}{\Gamma(\kappa+\frac{1}{2})}\sqrt{\frac{2}{\pi x}}\int_{0}^{\infty}e^{-y}\left (y\left ( 1+\frac{iy}{2x} \right )    \right ) ^{\kappa-1/2} dy.
    \end{align*}
    When $\kappa$ is a positive integer, we derive (using the Taylor expansion for $J_{\kappa}(x)$ if $0<x\leq 1$ , or the above integral expression for W(x) if $x\geq 1$) the following bounds for the derivatives of $W$
    \begin{equation}
        x^jW^{(j)}(x)\ll \frac{x}{(1+x)^{3/2}}
    \end{equation}
     for any $j\geq 0$, the implied constant depending on $j$ and $\kappa$.
\end{lemma}

\subsection{Oscillatory integrals}
Let $\mathcal{F}$ be an index set and $X=X_T: \mathcal{F} \rightarrow \mathbb{R}_{\geq 1}$ be a function of $T \in \mathcal{F}$.
 A family $\left\{w_T\right\}_{T \in \mathcal{F}}$ of smooth functions supported on a product of dyadic intervals in $\mathbb{R}_{>0}^d$ is called $X$-inert if for each $j=\left(j_1, \cdots, j_d\right) \in \mathbb{Z}_{\geq 0}^d$ we have
\begin{equation}\label{eqn: def_inert}
    C_{\mathcal{F}}\left(j_1, \cdots, j_d\right):=\sup _{T \in \mathcal{F}} \sup _{\left(x_1, \cdots, x_d\right) \in \mathbb{R}_{>0}^d} X_T^{-j_1-\cdots-j_d}\left|x_1^{j_1} \cdots x_d^{j_d} w_T^{\left(j_1, \cdots, j_d\right)}\left(x_1, \ldots, x_d\right)\right|<\infty .
\end{equation}
We will frequently show that oscillatory integrals are very small using integration by parts and analyze  their main contributions. We quote here some useful lemmas from \cite[Lemma 8.1]{BKY13} and \cite[Main Theorem]{KPY19}.
\begin{lemma}\label{lemma: spl_1}
 Suppose that $w=w_T(t)$ is a family of $X$-inert functions, with compact support on $[Z, 2 Z]$, so that for all $j=0,1, \ldots$ we have the bound $w^{(j)}(t) \ll(Z / X)^{-j}$. Also suppose that $\phi$ is smooth and satisfies, for $j=2,3, \cdots$, $\phi^{(j)}(t) \ll \frac{Y}{Z^j}$ for some $R \geq 1$ with $Y / X \geq R$ and all $t$ in the support of $w$. Let
    \begin{equation}
    I=\int_{-\infty}^{\infty} w(t) e^{i \phi(t)} \dd t.
    \end{equation}
    If $\left|\phi^{\prime}(t)\right| \gg \frac{Y}{Z}$ for all $t$ in the support of $w$, then $I \ll_A Z R^{-A}$ for $A$ arbitrarily large.
\end{lemma}
\begin{proof}
    See \cite{BKY13} and \cite{KPY19}.
\end{proof}
\begin{lemma}\label{lemma: spl_2}
  Suppose $w_T$ is $X$-inert in $t_1, \cdots, t_d$, supported on $t_1 \asymp Z$ and $t_i \asymp X_i$ for $i=2, \cdots, d$. Suppose that on the support of $w_T, \phi=\phi_T$ satisfies
\[
\frac{\partial^{a_1+a_2+\cdots+a_d}}{\partial t_1^{a_1} \cdots \partial t_d^{a_d}} \phi\left(t_1, t_2, \cdots, t_d\right) \ll_{C_{\mathcal{F}}} \frac{Y}{Z^{a_1}} \frac{1}{X_2^{a_2} \cdots X_d^{a_d}},
\]
for all $a_1, \cdots, a_d \in \mathbb{N}$ with $a_1 \geq 1$. Suppose $\phi^{\prime \prime}\left(t_1, t_2, \ldots, t_d\right) \gg \frac{Y}{Z^2}$ (here and later, $\phi^{\prime}$ and $\phi^{\prime \prime}$ denote the derivative with respect to $\left.t_1\right)$, for all $t_1, t_2, \cdots, t_d$ in the support of $w_T$, and for each $t_2, \cdots, t_d$ in the support of $\phi$ there exists $t_0 \asymp Z$ such that $\phi^{\prime}\left(t_0, t_2, \ldots, t_d\right)=0$. Suppose that $Y / X^2 \geq R$ for some $R \geq 1$. Then
\begin{equation}\label{eqn: spm}
I=\int_{\mathbb{R}} e^{i \phi\left(t_1, \cdots, t_d\right)} w_T\left(t_1, \cdots, t_d\right) d t_1=\frac{Z}{\sqrt{Y}} e^{i \phi\left(t_0, t_2, \cdots, t_d\right)} W_T\left(t_2, \cdots, t_d\right)+O_A\left(Z R^{-A}\right),
\end{equation}
for some $X$-inert family of functions $W_T$, and where $A>0$ may be taken to be arbitrarily large. The implied constant in equation \eqref{eqn: spm} depends only on $A$ and on $C_{\mathcal{F}}$ defined in formula \eqref{eqn: def_inert}.
\begin{proof}
    See \cite{BKY13} and \cite{KPY19}.
\end{proof}
\end{lemma}
\medskip
\textbf{Notation.}
Throughout the paper, $\varepsilon$ is an arbitrarily small positive number that may change from line to line.
We use $y\asymp Y$ to denote that $c_1 Y\leq y\leq c_2 Y$ for some positive constants $c_1$ and $c_2$. We say the contribution is negligible if it is $\ll_B q^{-B}$ for an arbitrary constant $B>0$.

\section{Applying the Petersson trace formula}\label{Applying the Petersson trace formula}
Using the approximate functional equations, we get
\begin{align*}
    \sum_{f\in \mathcal{B}_k{(q,\chi)}}&\omega_f^{-1}L(1/2,f)L(1/2,\sym^2f)=S_1+S_2+O(q^{-100}),
\end{align*}
where
\begin{align*}
&S_1=2\sum_{f\in \mathcal{B}_k{(q,\chi)}}\omega_f^{-1}\sum_{m=1}^{\infty}\frac{\lambda_f(m)}{\sqrt{m}}V\left(\frac{m}{\sqrt{q}}\right)\sum_{n=1}^{\infty}\frac{\lambda_f(n^2)}{\sqrt{n}}W\left(\frac{n}{\sqrt{q}}\right),\\
&S_2=2\sum_{f\in \mathcal{B}_k{(q,\chi)}}\omega_f^{-1}i^{k}\overline{\lambda}_f(q)\overline{\tau}_{\chi}q^{-\frac{1}{2}}\sum_{m=1}^{\infty}\frac{\overline{\lambda_f(m)}}{\sqrt{m}}V\left(\frac{m}{\sqrt{q}}\right)\sum_{n=1}^{\infty}\frac{\lambda_f(n^2)}{\sqrt{n}}W\left(\frac{n}{\sqrt{q}}\right).
\end{align*}
By Hecke relation \eqref{heckerelation}, $\lambda_{f}(q)\lambda_{f}(m)=\lambda_{f}(qm)$, and by \eqref{dualoflambda}, $\lambda_{f}(m)=\chi(m)\overline{\lambda_{f}(m)}$. Thus,
\begin{align*}
    S_1=2\sum_{m=1}^{\infty}\frac{\chi(m)}{\sqrt{m}}V\left(\frac{m}{\sqrt{q}}\right)\sum_{n=1}^{\infty}\frac{1}{\sqrt{n}}W\left(\frac{n}{\sqrt{q}}\right)\sum_{f\in \mathcal{B}_k{(q,\chi)}}\omega_f^{-1}\overline{\lambda_f(m)}\lambda_f(n^2),
\end{align*}
\begin{align*}
    S_2=2i^{k}\overline{\tau}_{\chi}q^{-\frac{1}{2}}\sum_{m=1}^{\infty}\frac{1}{\sqrt{m}}V\left(\frac{m}{\sqrt{q}}\right)\sum_{n=1}^{\infty}\frac{1}{\sqrt{n}}W\left(\frac{n}{\sqrt{q}}\right)\sum_{f\in \mathcal{B}_k{(q,\chi)}}\omega_f^{-1}\overline{\lambda_f(qm)}\lambda_f(n^2).
\end{align*}

Then applying the Petersson trace formula, we get
\begin{align*}
    &S_1=\mathcal{D}_1+2\pi i^{-k}\mathcal{O}_1,\\
    &S_2=\mathcal{D}_2+2\pi i^{-k}\mathcal{O}_2,
\end{align*}
where
\begin{align}\label{D1}
    \mathcal{D}_1=2\sum_{n=1}^{\infty}\frac{\chi(n^2)}{n^{3/2}}V\left(\frac{n^2}{\sqrt{q}}\right)W\left(\frac{n}{\sqrt{q}}\right),
\end{align}
\begin{align}\label{D2}
\mathcal{D}_2=2i^{k}\overline{\tau_{\chi}}\sum_{\substack{n=1\\n^2\equiv0\ (\mod q)}}^{\infty}\frac{1}{n^{3/2}}V\left(\frac{n^2}{q^{3/2}}\right)W\left(\frac{n}{\sqrt{q}}\right)
\end{align}
are the diagonal terms respectively and 
\begin{align}\label{O1}
    \mathcal{O}_1=2\sum_{m=1}^{\infty}\frac{\chi(m)}{\sqrt{m}}V\left(\frac{m}{\sqrt{q}}\right)\sum_{n=1}^{\infty}\frac{1}{\sqrt{n}}W\left(\frac{n}{\sqrt{q}}\right)\sum_{c=1}^{\infty}\frac{S_{\chi}(m,n^2;cq)}{cq}J_{k-1}\left(\frac{4\pi \sqrt{m}n}{cq}\right),
\end{align}
\begin{align}\label{O2}
\mathcal{O}_2=2i^k\overline{\tau_\chi}q^{-1/2}\sum_{m=1}^{\infty}\frac{1}{\sqrt{m}}V\left(\frac{m}{\sqrt{q}}\right)\sum_{n=1}^{\infty}\frac{1}{\sqrt{n}}W\left(\frac{n}{\sqrt{q}}\right)\sum_{c=1}^{\infty}\frac{S_{\chi}(mq,n^2;cq)}{cq}J_{k-1}\left(\frac{4\pi \sqrt{m}n}{c\sqrt{q}}\right)
\end{align}
are the off-diagonal terms respectively.

\section{The diagonal term}\label{The diagonal term}
In this section, we deal with the diagonal terms $\mathcal{D}_1$ and $\mathcal{D}_2$. For $\mathcal{D}_1$, by definition of $V$ and $W$, we get
\begin{align*}
     \mathcal{D}_1=2\frac{1}{(2\pi i)^2}\int_{(1)}\int_{(1)}\frac{L_{\infty}(1/2+u, f)}{L_{\infty}(1/2, f)}\frac{L_{\infty}(1/2+v,\sym^2 f)}{L_{\infty}(1/2,\sym^2 f)}q^{\frac{u+v}{2}}\times\\
     \zeta^{(q)}(3/2+2u+v)\zeta^{(q)}(1+2v)\frac{G_1(u)}{u}\frac{G_2(v)}{v}dudv.
\end{align*}
Moving the line of the integration of $u$ to $\Re u=-\varepsilon$ and picking a residue at $u=0$, we get
\begin{align*}
    \mathcal{D}_1&= \mathcal{R}_1+\mathcal{I}_1,
\end{align*}
where
\begin{align*}
        \mathcal{R}_1= 2\frac{1}{2\pi i}\int_{(1)}\frac{L_{\infty}(1/2+v,\sym^2 f)}{L_{\infty}(1/2,\sym^2 f)}q^{\frac{v}{2}}\zeta^{(q)}(3/2+v)\zeta^{(q)}(1+2v)\frac{G_2(v)}{v}dv,         
\end{align*}
\begin{align*}
        \mathcal{I}_1=2\frac{1}{(2\pi i)^2}\int_{(1)}\int_{(-\varepsilon)}\frac{L_{\infty}(1/2+u, f)}{L_{\infty}(1/2, f)}\frac{L_{\infty}(1/2+v,\sym^2 f)}{L_{\infty}(1/2,\sym^2 f)}q^{\frac{u+v}{2}}\times\\
        \zeta^{(q)}(3/2+2u+v)\zeta^{(q)}(1+2v)\frac{G_1(u)}{u}\frac{G_2(v)}{v}dudv.
\end{align*}

Expanding the factors that appear above into Laurent series, we get
\begin{align*}
    &\frac{L_{\infty}(1/2+v,\sym^2 f)}{L_{\infty}(1/2,\sym^2 f)}\zeta^{(q)}(1+2v)\frac{G_2(v)}{v}\\
    &=\left(\frac{1}{v}+ \frac{{L}'_{\infty}(1/2,\sym^2 f)}{L_{\infty}(1/2,\sym^2 f)}+\cdots\right) \left (  (1-1/q)\frac{1}{2v}+(1-1/q)\left (  \gamma+\frac{\log q}{q-1}\right )+\cdots\right )    \\
    &=(1-1/q)\frac{1}{2v^2}+\left((1-1/q)\left(  \gamma+\frac{\log q}{q-1}\right)+\frac{(1-1/q)}{2}\frac{{L}'_{\infty}(1/2,\sym^2 f)}{L_{\infty}(1/2,\sym^2 f)}\right)\frac{1}{v}+\cdots ,       
    \end{align*}
    \begin{align*}
    q^{\frac{v}{2}}\zeta^{(q)}(3/2+v)=\zeta^{(q)}(3/2)+\left(\frac{\log q}{2}\zeta^{(q)}(3/2)+{\zeta^{(q)}}'(3/2)\right)v+\cdots.
\end{align*}
For $\mathcal{R}_1$, moving the line of integration to $\Re v=-1/2+\varepsilon$, picking up a double pole at $v=0$, we get
\begin{align*}
\mathcal{R}_1=\mathcal{R}_{0}+\mathcal{I}_{-1/2+\varepsilon},
\end{align*}
where
\begin{align*}
    \mathcal{R}_{0}=2\zeta^{(q)}(3/2)\left [ (1-1/q)\left(  \gamma+\frac{\log q}{q-1}\right)+\frac{(1-1/q)}{2}\frac{{L}'_{\infty}(1/2,\sym^2 f)}{L_{\infty}(1/2,\sym^2 f)} \right ] \\
    +(1-1/q)\left[\frac{\log q}{2}\zeta^{(q)}(3/2)+{\zeta^{(q)}}'(3/2)\right],
\end{align*}
and
\begin{align*}
    \mathcal{I}_{-1/2+\varepsilon}= 2\frac{1}{2\pi i}\int_{(-1/2+\varepsilon)}\frac{L_{\infty}(1/2+v,\sym^2 f)}{L_{\infty}(1/2,\sym^2 f)}q^{\frac{v}{2}}\zeta^{(q)}(3/2+v)\zeta^{(q)}(1+2v)\frac{G_2(v)}{v}dv\ll q^{-1/4+\varepsilon}.
\end{align*}

For $\mathcal{I}_1$, moving the integration of $v$ to $\Re v=-1/2+\varepsilon$, we can pick up a double pole at $v=0$ and a simple pole at $v=-1/2-2u$. Thus 
\begin{align*}  \mathcal{I}_1=\mathcal{R}_{1,v=0}+\mathcal{R}_{1,v=-1/2-2u}+\mathcal{I}_{-\varepsilon,-1/2+\varepsilon},
\end{align*}
where 
\begin{align*}
   &\mathcal{R}_{1,v=0}=2\frac{1}{2\pi i}\int_{(-\varepsilon)}\frac{L_{\infty}(1/2+u, f)}{L_{\infty}(1/2, f)}q^{\frac{u}{2}}H(u)\frac{G_1(u)}{u}du,\\
   &H(u)=1/2{\zeta^{(q)}}'(3/2+2u)+(1-1/q)\left(\gamma+\frac{\log q}{q-1}\right)\zeta^{(q)}(3/2+2u)\\
   &+\frac{1}{2}\zeta^{(q)}(3/2+2u)\left(\frac{\log q}{2}+\frac{{L}'_{\infty}(1/2,\sym^2 f)}{L_{\infty}(1/2,\sym^2 f)}\right),
\end{align*}
\begin{align*}
    \mathcal{R}_{1,v=-1/2-2u}&=2(1-1/q)\frac{1}{2\pi i}\int_{(-\varepsilon)}\frac{L_{\infty}(1/2+u, f)}{L_{\infty}(1/2, f)}\frac{L_{\infty}(-2u,\sym^2 f)}{L_{\infty}(1/2,\sym^2 f)}q^{-\frac{u}{2}-1/4}\times\\
    &\zeta^{(q)}(-4u)\frac{G_1(u)}{u}\frac{G_2(-1/2-2u)}{-1/2-2u}du,
\end{align*}
\begin{align*}
        \mathcal{I}_{-\varepsilon,-1/2+\varepsilon}=2\frac{1}{(2\pi i)^2}\int_{(-\varepsilon)}\int_{(-1/2+\varepsilon)}\frac{L_{\infty}(1/2+u, f)}{L_{\infty}(1/2, f)}\frac{L_{\infty}(1/2+v,\sym^2 f)}{L_{\infty}(1/2,\sym^2 f)}q^{\frac{u+v}{2}}\times\\
        \zeta^{(q)}(3/2+2u+v)\zeta^{(q)}(1+2v)\frac{G_1(u)}{u}\frac{G_2(v)}{v}dudv.
\end{align*}

For $\mathcal{R}_{1,v=0}$, move the line of $u$ integration to $\Re u=-1/2+\varepsilon$, the double zero of $G_1(u)$ cancels the double pole of $H(u)$ at $u=-1/4$, thus 
\begin{align}
    \mathcal{R}_{1,v=0}\ll q^{-1/4+\varepsilon}.
\end{align}
For $\mathcal{R}_{1,v=-1/2-2u}$, we have
\begin{align}
    \mathcal{R}_{1,v=-1/2-2u}\ll q^{-1/4+\varepsilon}.
\end{align}
For $\mathcal{I}_{-\varepsilon,-1/2+\varepsilon}$, we have
\begin{align}
    \mathcal{I}_{-\varepsilon,-1/2+\varepsilon}\ll q^{-1/4+\varepsilon}.
\end{align}

Recall that $\zeta^{(q)}(s)=(1-q^{-s})\zeta(s)$, we have
\begin{align}\label{contribution of D1}
\mathcal{D}_1 =&\frac{\zeta(3/2)}{2}\log q + \zeta(3/2)\left( 2\gamma - \log 2 - \frac{3}{2}\log \pi + \frac{{\zeta}'(3/2)}{\zeta(3/2)} \right.\nonumber\\
&\left.+ \psi\left(k-\frac{1}{2}\right) + \frac{1}{2}\psi\left(\frac{3}{4}-\frac{\nu}{2}\right) \right) + O(q^{-1/4+\varepsilon}).
\end{align}

For $\mathcal{D}_2$, recalling \eqref{D2}, the condition $mq=n^2$ implies $q\mid n^2$. Since $q$ is prime, we have $q\mid n$. Let $n=qr$, we get 
\begin{align*}
\mathcal{D}_2=2i^{k}\overline{\tau_{\chi}}\sum_{r=1}^{\infty}\frac{1}{r^{3/2}q^{3/2}}V\left(\frac{r^2q^2}{q^{3/2}}\right)W\left(\frac{qr}{\sqrt{q}}\right),
\end{align*}
 which is negligible as the functions $W$ and $V$ rapidly decay.

We conclude that the diagonal contribution is \eqref{contribution of D1}.

\section{The off-diagonal term}\label{The off-diagonal term}
In this section, we deal with the off-diagonal terms $\mathcal{O}_1$ and $\mathcal{O}_2$.
\subsection{Analysis of $\mathcal{O}_1$}
For $\mathcal{O}_1$, by the rapid decay of $V$ and $W$, we can restrict the sums to $m\ll q^{1/2+\varepsilon}$ and $n\ll q^{1/2+\varepsilon}$ up to a negligible error. Recall that $q$ is a prime, which yields $(nm,q)=1$; we have $(m,n^2,cq)^{1/2}\leq (m,n^2,c)^{1/2}\leq c^{1/2}$. Thus by Weil's bound, we have
\begin{align}\label{Weilbound}
    S_{\chi}(m,n^2;cq)\ll (cq)^{1/2}\tau(cq)(m,n^2,cq)^{1/2}\ll c^{1+\varepsilon}q^{1/2+\varepsilon}.
\end{align}
Using the bound $J_{\kappa}(x)\ll x^{\kappa}$ and \eqref{Weilbound}, recalling that $k\geq 4$, we see that 
\begin{align}\label{contribution of O1}
    \mathcal{O}_1\ll \sum_{m\ll q^{1/2+\varepsilon}}\frac{1}{\sqrt{m}}\sum_{n\ll q^{1/2+\varepsilon}}\frac{1}{\sqrt{n}}\sum_{c=1}^{\infty}\frac{c^{1+\varepsilon}q^{1/2+\varepsilon}}{cq}\left(\frac{\sqrt{m}n}{cq}\right)^{k-1}\nonumber\\
    \ll q^{1/4+\varepsilon} \cdot q^{1/4+\varepsilon}\cdot q^{-1/2+\varepsilon} q^{-\frac{1}{4}(k-1)+\varepsilon}\ll q^{-3/4+\varepsilon}.
\end{align}

\subsection{Analysis of $\mathcal{O}_2$}
We take a smooth dyadic subdivision of the $n$ and $m$ sums in \eqref{O2}. This leads us to sums of type
\begin{align}\label{offdiagonalterm2}
(NM)^{-1/2}\sum_{m=1}^{\infty}\mathcal{V}\left(\frac{m}{M}\right)\sum_{n=1}^{\infty}\mathcal{W}\left(\frac{n}{N}\right)\sum_{c=1}^{\infty}\frac{S_{\chi}(mq,n^2;cq)}{cq}J_{k-1}\left(\frac{4\pi \sqrt{m}n}{c\sqrt{q}}\right).
\end{align}

Here $M\ll q^{1/2+\varepsilon}$, $N\ll q^{1/2+\varepsilon}$ and the functions $\mathcal{V}$, $\mathcal{W}$ are supported in $[1,2]$, and satisfy the bound 
\begin{align*}
    \mathcal{V}^{(j)}(x), \mathcal{W}^{(j)}(x) \ll 1.
\end{align*}

Applying the Poisson summation to the $m$ sum with modulus $c$ yields
\begin{align*}
    \sum_{m=1}^{\infty}S_{\chi}(qm,n^2;cq)J_{k-1}\left(\frac{4\pi \sqrt{m}n}{c\sqrt{q}}\right)\mathcal{V}\left(\frac{m}{M}\right)=\frac{M}{c}&\sum_{m\in \mathbb{Z}}\sum_{a\ (\mod c)}S_{\chi}(aq,n^2;cq)e\left ( \frac{am}{c} \right ) \\
    &\times\int_{\mathbb{R}} \mathcal{V}(x)J_{k-1}\left(\frac{4\pi \sqrt{M}n\sqrt{x}}{c\sqrt{q}}\right)e\left ( -\frac{mMx}{c} \right ) dx.
\end{align*}

Write 
$$I:=\int_{\mathbb{R}} \mathcal{V}(x)J_{k-1}\left(\frac{4\pi \sqrt{M}n\sqrt{x}}{c\sqrt{q}}\right)e\left ( -\frac{mMx}{c} \right ) dx.$$ 
We will estimate the contribution to \eqref{offdiagonalterm2} from all $c$ in the dyadic range $C<c\leq 2C$.

For $C\gg q^{\delta-1/2}\sqrt{M}N$, i.e. $\frac{\sqrt{M}N}{c\sqrt{q}}\ll q^{-\delta}$, where $\delta>0$ is a small parameter to be chosen later, we use the bound $J_{\kappa}(x)\ll x^{\kappa}$. Integrating by parts shows that $I$ is negligible unless
\begin{align*}
  \left | m \right | \ll \frac{cq^{\varepsilon}}{M}.
\end{align*}

For such $m$, we have
\begin{align*}
        I\ll \left(\frac{\sqrt{M}n}{c\sqrt{q}}\right)^{k-1}.
\end{align*}

Since $n\asymp N\ll q^{1/2+\varepsilon}$ yields $(n,q)=1$, by Weil's bound, we get
\begin{align*}
    \sum_{a\ (\mod c)}S_{\chi}(aq,n^2;cq)e\left ( \frac{am}{c} \right )&\ll \sum_{a\ (\mod c)}(aq,n^2,cq)^{1/2}(cq)^{1/2}\tau(cq)\ll\sum_{a\ (\mod c)}(a,n^2,c)^{1/2}(cq)^{1/2+\varepsilon}\\
    &\ll (cq)^{1/2+\varepsilon}\sum_{d \mid c} \sum_{\substack{a\ (\mod c)\\ d\mid a}}d^{1/2}\ll(cq)^{1/2+\varepsilon} \sum_{d\mid c}\frac{c}{d^{1/2}}\ll c^{3/2+2\varepsilon}q^{1/2+\varepsilon}. 
\end{align*}
Thus the contribution of this part is 
\begin{align}\label{besselsmall}
&\ll (MN)^{-1/2}\sum_{n\geq 1} \mathcal{W}\left(\frac{n}{N}\right)\sum_{c\sim C}\frac{c^{3/2+\varepsilon}q^{1/2+\varepsilon}}{cq}\frac{M}{c}\sum_{\left | m \right |\ll \frac{cq^{\varepsilon}}{M}}\left(\frac{\sqrt{M}n}{c\sqrt{q}}\right)^{k-1} \nonumber\\
&\ll \frac{M^{1/4}N^3q^{\varepsilon}}{q^{5/4}}q^{-\delta(k-5/2)}\ll q^{-1/8-\delta(k-5/2)}.
\end{align}

Now we assume that $C\ll q^{\delta-1/2}\sqrt{M}N$. In this range, we will show the arithmetic part has square root cancellation and use stationary phase method to analyze the analytic part which includes the oscillation integral $I$.

We start with the arithmetic part. In particular we have $(c,q)=1$ since we assume $q$ is a prime. We need to evaluate the character sum $$\mathfrak{S}:=\sum_{a\ (\mod c)}S_{\chi}(qa,n^2;cq)e\left(\frac{am}{c}\right).$$

Opening the Kloosterman sum, we see that
\begin{align*}
\mathfrak{S}=\sum_{a\ (\mod c)}\sum_{\substack{d\ (\mod cq)\\(d,cq)=1}}\chi(d)e\left(\frac{qad+n^2\bar{d}}{cq}+\frac{am}{c}\right)&=\sum_{\substack{d\ (\mod cq)\\(d,cq)=1}}\chi(d)e\left(\frac{n^2\bar{d}}{cq}\right)\sum_{a\ (\mod c)}e\left(\frac{a(m+d)}{c}\right)\\
&=c\sum_{\substack{d\ (\mod cq)\\(d,cq)=1\\c\mid d+m}}\chi(d)e\left(\frac{n^2\bar{d}}{cq}\right).
\end{align*}

By the Chinese Remainder Theorem, we can write $d\ (\mod cq)$ as $d_cq\bar{q}+d_qc\bar{c}$, where $d_c \ (\mod c)$, $d_q \ (\mod q)$, $q\bar{q}\equiv 1\ (\mod c)$ and $c\bar{c}\equiv 1\ (\mod q)$, since $(c,q)=1$. Then the condition $c\mid d+m$ implies $d_c\equiv-m\ (\mod c)$. By $(d,cq)=1$, we get $(m,c)=1$, so $\bar{m}\ (\mod c)$ exists. Recall that $\chi$ is real, we have
\begin{align}\label{src}
    \mathfrak{S}&=c\sum_{\substack{d_q\ (\mod q)\\(d_q,q)=1}}\chi(d_q)e\left(\frac{n^2\left(\bar{d_q}c\bar{c}-\bar{m}q\bar{q}\right)}{cq}\right)\nonumber\\
    &=ce\left(\frac{-n^2\bar{m}\bar{q}}{c}\right)\sum_{\substack{d_q\ (\mod q)\\(d_q,q)=1}}\chi(d_q)e\left(\frac{n^2\bar{d_q}\bar{c}}{q}\right)=ce\left(\frac{-n^2\bar{m}\bar{q}}{c}\right)\chi(n^2\bar{c})\tau(\bar{\chi})\ll cq^{1/2}.
\end{align}

For $I$, by lemma \ref{besselfunction}, we can write $I=I_{+}+I_{-}$, where 
$$I_{+}=\int_{\mathbb{R}} \mathcal{V}(x)W\left(\frac{4\pi \sqrt{M}n\sqrt{x}}{c\sqrt{q}}\right)e\left (  \frac{2 \sqrt{M}n\sqrt{x}}{c\sqrt{q}} -\frac{mMx}{c} \right ) dx,$$
and 
$$I_{-}=\int_{\mathbb{R}} \mathcal{V}(x)\overline{W}\left(\frac{4\pi \sqrt{M}n\sqrt{x}}{c\sqrt{q}}\right)e\left ( -\frac{2 \sqrt{M}n\sqrt{x}}{c\sqrt{q}} -\frac{mMx}{c} \right ) dx,$$

Let 
$$z(x)=\frac{4\pi \sqrt{M}n\sqrt{x}}{c\sqrt{q}},
$$  
$$\phi_{+}(x)=\frac{4\pi \sqrt{M}n\sqrt{x}}{c\sqrt{q}} -\frac{2\pi mMx}{c}$$
and 
$$\phi_{-}(x)=-\frac{4\pi \sqrt{M}n\sqrt{x}}{c\sqrt{q}} -\frac{2\pi mMx}{c}.$$

We have 
\begin{align*}
    &{\phi}'_{+}(x)=\frac{2\pi \sqrt{M}n}{c\sqrt{qx}}-\frac{2\pi Mm}{c},\\
    &{\phi}''_{+}(x)=-\frac{\pi \sqrt{M}n}{c\sqrt{q}}x^{-3/2} \asymp \frac{ \sqrt{M}n}{c\sqrt{q}},\\
    &{\phi}^{(j)}_{+}(x)\ll_{j} \frac{ \sqrt{M}n}{c\sqrt{q}}.
\end{align*}

Let ${\phi}'_{+}(x_0)=0$, we have $x_0=\frac{n^2}{m^2Mq}$.

\subsubsection{The range of small $c$: $C\ll q^{-\varepsilon-1/2}\sqrt{M}N$}

In this range we have $\frac{\sqrt{M}N}{c\sqrt{q}}\gg q^{\varepsilon}$.
\begin{itemize}
\item\textbf{Case I:}
If $x_0\in \left [  1,2\right ]$, i.e. $ \frac{1}{\sqrt{2}}\frac{n}{\sqrt{Mq}}\leq \left | m \right | \leq \frac{n}{\sqrt{Mq}}$, then we take $X=1,\ Z=1$, $Y=\frac{\sqrt{M}n}{c\sqrt{q}}+\frac{M \left | m  \right | }{c}$, we have $Y/X^2\geq \frac{\sqrt{M}n}{c\sqrt{q}}>1$. Thus by lemma \ref{lemma: spl_2} with $R=\frac{\sqrt{M}n}{c\sqrt{q}}$, we have
\begin{align*}
    I_{+}\asymp\frac{e^{\frac{2\pi in^2}{cmq}}}{\sqrt{Y}}F_T(x_0)+O_{A}(Y^{-A})\ll \frac{c^{1/2}q^{1/4}}{M^{1/4}N^{1/2}}.
\end{align*}
The contribution of this part to \eqref{offdiagonalterm2} is
\begin{equation}\label{stationaryI+1}
    \ll(MN)^{-1/2}\sum_{n\geq 1} \mathcal{W}(\frac{n}{N})\sum_{c\sim C}\frac{cq^{1/2}}{cq}\frac{M}{c}\sum_{\left | m \right |\asymp\frac{n}{\sqrt{Mq}}}\frac{c^{1/2}q^{1/4}}{M^{1/4}N^{1/2}}\ll \frac{C^{1/2}N}{M^{1/4}q^{3/4}}\ll q^{-1/4+\varepsilon}.
\end{equation}

\item\textbf{Case II:} If $x_0<1$, i.e. $\left | m \right | > \frac{n}{\sqrt{Mq}}$, we can use lemma \ref{lemma: spl_1}. Take $X=1$, $Z=1$, $Y=\frac{Mm}{c}$, we have $Y/X^2\geq \frac{\sqrt{M}n}{c\sqrt{q}}>q^{\varepsilon}>1$. Thus by lemma \ref{lemma: spl_1} with $R=\frac{Mm}{c}$, we have
\begin{equation*}
    I_{+}\ll \left(\frac{Mm}{c}\right)^{-A}.
\end{equation*}
The contribution of this part to \eqref{offdiagonalterm2} is
\begin{equation}\label{outstationaryI+1}
    \ll(MN)^{-1/2}q^{\varepsilon}\sum_{n\geq 1} \mathcal{W}(\frac{n}{N})\sum_{c\sim C}\frac{cq^{1/2}}{cq}\frac{M}{c}\sum_{\left | m \right |>\frac{n}{\sqrt{Mq}}}\left(\frac{Mm}{c}\right)^{-A}\ll q^{-(A-1)\varepsilon}N^{3/2}.
\end{equation}

Thus the contribution is negligible.

\item\textbf{Case III:} If $x_0>2$, i.e. $\left | m \right | < \frac{n}{\sqrt{2Mq}}$, we can use lemma \ref{lemma: spl_1}. Take $X=1$, $Z=1$, $Y=\frac{\sqrt{M}n}{c\sqrt{q}}$, we have $Y/X^2\geq \frac{\sqrt{M}n}{c\sqrt{q}}>q^{\varepsilon}>1$. Thus by lemma \ref{lemma: spl_1} with $R=\frac{\sqrt{M}n}{c\sqrt{q}}$, we have
\begin{equation*}
    I_{+}\ll \left(\frac{\sqrt{M}n}{c\sqrt{q}}\right)^{-A}.
\end{equation*}
The contribution of this part to \eqref{offdiagonalterm2} is
\begin{equation}\label{outstationaryI+2}
    \ll(MN)^{-1/2}\sum_{n\geq 1} \mathcal{W}(\frac{n}{N})\sum_{c\sim C}\frac{cq^{1/2}}{cq}\frac{M}{c}\sum_{\left | m \right | \ll \frac{n}{\sqrt{Mq}}}\left(\frac{\sqrt{M}n}{c\sqrt{q}}\right)^{-A}\ll N^{3/2}q^{-1-(A-1)\varepsilon}.
\end{equation}

Thus the contribution is negligible.
\end{itemize}
Thus the total contribution from the small $c$ range is $O(q^{-1/4+\varepsilon})$.

\subsubsection{The range of intermediate $c$: $q^{-\varepsilon-1/2}\sqrt{M}N\ll C\ll q^{\delta-1/2}\sqrt{M}N$}

In this range we have $q^{-\delta}\ll \frac{\sqrt{M}N}{c\sqrt{q}}\ll q^{\varepsilon}$.
\begin{itemize}
    \item \textbf{Case I:}
If $x_0\in \left [  1,2\right ]$, i.e. $ \frac{1}{\sqrt{2}}\frac{n}{\sqrt{Mq}}\leq \left | m \right | \leq \frac{n}{\sqrt{Mq}}$, we use the trivial bound of the Bessel function. Combining with \eqref{src}, the contribution of this part is 
\begin{align}\label{stationaryI+2}
    \ll(MN)^{-1/2}\sum_{n\geq 1} \mathcal{W}(\frac{n}{N})\sum_{c\sim C}\frac{cq^{1/2}}{cq}\frac{M}{c}\sum_{\left | m \right |\asymp \frac{n}{\sqrt{Mq}}}1\ll \frac{N^{3/2}}{q}q^{\varepsilon}\ll q^{-1/4+\varepsilon}.
\end{align}

\item\textbf{Case II:} If $x_0<1$, i.e. $\left | m \right | > \frac{n}{\sqrt{Mq}}$, then ${\phi}'_{+}(x)\gg \frac{M\left | m \right |}{c}$. Thus by lemma \ref{lemma: spl_1}, $I$ is negligible unless $\left | m \right |\ll  \frac{cq^{\varepsilon}}{M}$.
Using trivial bound $I_{+}\ll1$, we get the contribution of this part to \eqref{offdiagonalterm2} is
\begin{equation}\label{outstationaryI+22}
    \ll(MN)^{-1/2}q^{\varepsilon}\sum_{n\geq 1} \mathcal{W}(\frac{n}{N})\sum_{c\sim C}\frac{cq^{1/2}}{cq}\frac{M}{c}\sum_{\left | m \right |\ll\frac{cq^{\varepsilon}}{M}}1\ll \frac{N^{3/2}q^{\delta+\varepsilon}}{q}\ll q^{-1/4+\delta+\varepsilon}.
\end{equation}

\item\textbf{Case III:} If $x_0>2$, i.e. $\left | m \right | < \frac{n}{\sqrt{2Mq}}$.
Using trivial bound $I_{+}\ll 1$, we get the contribution of this part to \eqref{offdiagonalterm2} is
\begin{equation}\label{outstationaryI+23}
    \ll(MN)^{-1/2}q^{\varepsilon}\sum_{n\geq 1} \mathcal{W}(\frac{n}{N})\sum_{c\sim C}\frac{cq^{1/2}}{cq}\frac{M}{c}\sum_{\left | m \right |\ll\frac{n}{\sqrt{Mq}}}1\ll  \frac{N^{3/2}q^{\varepsilon}}{q}\ll q^{-1/4+\varepsilon}.
\end{equation}
\end{itemize}

Hence the intermediate $c$ range contributes $O(q^{-1/4+\delta+\varepsilon})$.

By \eqref{stationaryI+1}, \eqref{outstationaryI+1}, \eqref{outstationaryI+2}, \eqref{stationaryI+2}, \eqref{outstationaryI+22} and \eqref{outstationaryI+23}, we see that the contribution of $I_{+}$ to \eqref{offdiagonalterm2} is 
\begin{equation}\label{bessellarge}
    \ll q^{-1/4+\delta+\varepsilon} .
\end{equation}

By a similar argument, we could prove the same upper bound for $I_{-}$. Thus by \eqref{besselsmall} and  \eqref{bessellarge}, we see that \eqref{offdiagonalterm2} is 
\begin{equation}\label{contribution of O2}
    \ll \max\left \{  q^{-1/8-\delta(k-5/2)},q^{-1/4+\delta+\varepsilon}\right \}.
\end{equation}

\section{Proof of Theorem \ref{main thm}}\label{Proof of main thm}
In this section, we need to choose a suitable $\delta$ to optimize the error term. Setting $-1/8-\delta(k-5/2)=-1/4+\delta$ gives $\delta=\delta_k:=\frac{1}{8(k-3/2)}$. Together with \eqref{contribution of D1}, \eqref{contribution of O1}, \eqref{contribution of O2}, we conclude that
    \begin{equation}
        \sum_{f\in \mathcal{B}_k(q,\chi)}\omega_f^{-1}L(1/2,f)L(1/2,\sym^2f)=\frac{\zeta(3/2)}{2}\log q +C_k+O(q^{-\frac{1}{4}+\delta_k+\varepsilon}).
    \end{equation}

\section*{Acknowledgments}
 The author would like to thank Prof.\ Bingrong Huang for suggesting the problem and for his help and encouragement. He gratefully thanks the referees for the constructive comments and recommendations.

\addcontentsline{toc}{section}{References}
\phantomsection

\bibliographystyle{abbrv}
\bibliography{reference}

\end{document}